\documentclass[10pt, a4paper]{amsart}
\usepackage[hidelinks]{hyperref}
\usepackage{cleveref}

\usepackage{amssymb,amsmath,mathtools}
\usepackage{enumerate}
\usepackage{color}
\usepackage{mathrsfs}
\usepackage{enumitem}
\usepackage{graphicx}
\usepackage{float}
\usepackage{tikz}
\usepackage{mathtools}

\usepackage{setspace}

\usepackage[a4paper, left=25mm, right=25mm, top=25mm, bottom=25mm]{geometry}

\newtheorem{theorem}{Theorem}[section]
\newtheorem{proposition}[theorem]{Proposition}
\newtheorem{lemma}[theorem]{Lemma}
\newtheorem{corollary}[theorem]{Corollary}
\theoremstyle{definition}

\DeclareFontFamily{U}{mathx}{}
\DeclareFontShape{U}{mathx}{m}{n}{ <-> mathx10 }{}
\DeclareSymbolFont{mathx}{U}{mathx}{m}{n}
\DeclareFontSubstitution{U}{mathx}{m}{n}

\DeclareMathAccent{\widecheck}{0}{mathx}{"71}

\makeatletter
\newcommand{\wcheck}[1]{\mathpalette\wcheck@{#1}}
\newcommand{\wcheck@}[2]{%
\begingroup
\edef\wcheck@font{\the
\ifx#1\displaystyle\textfont\else\ifx#1\textstyle\textfont
\else\ifx#1\scriptstyle\scriptfont\else\scriptscriptfont\fi\fi\fi\@ne
}%
\sbox\z@{\wcheck@font\mbox{#2}\mbox{\char\the\skewchar\font}}%
\sbox\tw@{\wcheck@font#2\char\the\skewchar\font}%
\dimen@=\dimexpr\wd\tw@-\wd\z@\relax
{\,\kern2\dimen@\widecheck{\!\kern-2\dimen@#2\!}\,}%
\endgroup
}
\makeatother

\newcommand{\comment}[1]{}
\newcommand{\F}{\mathcal F}

\newcommand{\AW}{\mathcal A\mathcal W}

\newcommand{\N}{\mathbb N}
\newcommand{\R}{\mathbb R}
\newcommand{\X}{\mathcal X}

\newcommand{\cpl}{\mathrm{cpl}}
\newcommand{\cplc}{\cpl_{\mathrm{c}}}
\newcommand{\cplbc}{\cpl_{\mathrm{bc}}}

\newcommand{\dd}{\mathrm{d}}

\newcommand{\CW}{\mathcal C\mathcal W}
\newcommand{\SCW}{\mathcal S\mathcal C\mathcal W}

\def\P{{\mathbb P}}
\def\E{{\mathbb E}}

\def\R{{\mathbb R}}

\usepackage{relsize}
\def\fcmp{\mathbin{\raise 0.6ex\Hbox{\oalign{\cHfil$\scriptscriptstyle \mathrm{o}$\Hfil\cr\Hfil$\scriptscriptstyle\mathrm{9}$\Hfil}}}}

\numberwithin{equation}{section}

\author{Mathias Beiglb\"ock}
\address{Faculty of Mathematics, University of Vienna, Oskar-Morgenstern-Platz 1, 1090 Vienna, Austria}
\email{mathias.beiglboeck@univie.ac.at}
\thanks{\textit{Funding.} MB gratefully acknowledges financial support from FWF through P34743, P35197, and FW506064 and from OeNB through AB1898311.}
\author{Markus Zona}
\address{Faculty of Mathematics, University of Vienna, Oskar-Morgenstern-Platz 1, 1090 Vienna, Austria}
\email{a12209252@unet.univie.ac.at}
\date{}

\begin{document}
\title[Adapted Pinsker]{Pinsker's inequality for adapted total variation}

\begin{abstract} Pinsker's classical inequality asserts that the total variation $TV(\mu, \nu)$ between two probability measures is bounded by $\sqrt{ 2H(\mu|\nu)}$ where $H$ denotes the relative entropy (or Kullback-Leibler divergence). 
Considering the discrete metric, $TV$ can be seen as a Wasserstein distance and as such possesses causal / bicausal variants. Such adapted Wasserstein distances have distinct advantages over their classical counterparts when $\mu, \nu$ are the laws of stochastic processes $(X_k)_{k=1}^n, (Y_k)_{k=1}^n$ and exhibit numerous applications from stochastic control to machine learning. In this note we observe that the bicausal and causal versions of the total variation distance satisfy the Pinsker-type inequalities 
\begin{align*}
    ATV(\mu, \nu)& \leq \sqrt{n} \sqrt{2 H(\mu|\nu)}, \\
    \max\{\operatorname{CTV}(\mu,\nu),\operatorname{CTV}(\nu,\mu)\}& \leq 2 \sqrt{2 H(\mu|\nu)}.
\end{align*} 
In the bicausal case the constant is sharp for every $n$; in the causal case it is asymptotically sharp as $n\to\infty$.
\end{abstract}


\maketitle

\section{Introduction}%
Throughout $\mu, \nu$ denote probabilities on some Polish space $\X$. Their total variation distance 
is 
\begin{align*}
TV(\mu,\nu)= 
\int{\textstyle  \bigl| \frac{\dd \mu}{\dd(\mu+\nu)} - \frac{\dd \nu}{\dd(\mu+\nu)}   \bigr| }\, \dd (\mu+\nu) = 2 \sup_{A\subseteq \X, \text{mbl}} |\mu(A)-\nu(A)|.
\end{align*}
Alternatively $TV$ can be written as a Wasserstein distance
\begin{align}\label{eq:TVviaOT}
TV(\mu, \nu)= 2 \inf_{\pi\in \cpl(\mu, \nu)} \int \rho(x,y)\, \dd\pi(x,y),    
\end{align}
where $\rho(x,y)=I_{x\neq y}$ denotes the discrete metric and $\cpl(\mu, \nu)$ the set of couplings or transport plans  i.e.\ probabilities on $\X\times \X$ with marginals $\mu, \nu$. 
Pinsker's classical inequality asserts that $TV(\mu, \nu)\leq \sqrt{2H(\mu|\nu)}$, where $H$ denotes the relative entropy, i.e.\ $H(\mu|\nu) =\int\log(\dd \mu / \dd \nu)\, \dd\mu$ if $\mu\ll \nu$ and $H(\mu|\nu)= \infty $ otherwise. 

Our goal is to provide an \emph{adapted} version, which is of interest in the case where $\mu$ and $\nu$ are the laws of stochastic processes $(X_k)_{k=1}^n$ and $(Y_k)_{k=1}^n$, respectively. Here, $\X$ will be of product type, and it is notationally convenient to assume that $\X = \X_1 \times \ldots \times \X_n$, where each $\X_i$ is a Polish space in which $X_i,Y_i$, $i \leq n$, take values. When measuring the distance between stochastic processes, it is often useful to consider a variant of the Wasserstein distance that takes into account the inherent information structure of the processes. The resulting \emph{adapted / causal Wasserstein distance} has a variety of applications from stochastic optimization, stochastic control, and mathematical finance to the theory of geometric inequalities and machine learning; see \cite{PfPi14, La18, XuAc21, BaBePa21, AcKraPa24, CoLi24, JiOb24, BaWi23} and the references therein.
The difference to the usual Wasserstein distance is that one restricts to transport plans which are adapted / causal in the following sense: Let $(\F_k)_{k\leq n}$ be the canonical filtration on $\X$, i.e.\ generated by the first $k$ coordinates and consider a suitably chosen $\mu$-disintegration $(\pi_x)_x$ of $\pi\in \cpl(\mu, \nu)$. Then 
\begin{align}\label{eq:CausalDef}
\pi\text{ is \emph{causal}}:\Longleftrightarrow
x\mapsto\pi_x(B)\text{ is $\F_k$-measurable for $B\in\F_k$, $k\leq n$.}
\end{align}
When $\pi$ is of Monge-type, i.e.\ induced by some $T=(T_k)_{k=1}^n: \X\to \X$ it is causal if and only if $(T_k)_{k=1}^n$ is an adapted process, see \cite[Remark 2.3]{BaBeLiZa17}. 
A transport plan $\pi$ is \emph{bicausal}, in signs $\pi\in \cplbc(\mu, \nu)$, if \eqref{eq:CausalDef} holds for the $\mu$- and the $\nu$-disintegration. We write $\cplc(\mu,\nu)$ for the set of causal couplings. For a Borel metric $d$ on $\X$, the \emph{causal Wasserstein distance} $\CW_d$ and the \emph{adapted Wasserstein distance} $\AW_d$ are defined like the usual Wasserstein distance with cost $d$, but optimizing over causal and bicausal couplings, respectively.
The causal Wasserstein distance need not be symmetric. We also consider the \emph{symmetrized causal Wasserstein distance} $\SCW_d(\mu,\nu):=\max\{\CW_d(\mu,\nu),\CW_d(\nu,\mu)\}$.

In the Euclidean setting, the adapted Wasserstein distance has a geodesic completion by filtered processes \cite{BaBePa21}.
The symmetrized causal distance is equally suited for stability results \cite{BaBaBeEd19b} and can give tighter bounds, as below.

\noindent
For the discrete metric $\rho$, the \emph{adapted total variation} (see \cite{EcPa22}) is then given by
\begin{align*}
ATV(\mu,\nu):=2\inf_{\pi\in\cplbc(\mu,\nu)}\int\rho(x,y)\,\dd\pi(x,y)=2\AW_\rho(\mu,\nu),
\end{align*}
and its causal and symmetrized causal analogues are
\begin{align*}
    \operatorname{CTV}(\mu,\nu)&:=2\inf_{\pi\in\cplc(\mu,\nu)}\int\rho(x,y)\,\dd\pi(x,y)=2\CW_\rho(\mu,\nu), \\
    \operatorname{SCTV}(\mu,\nu)&:=2\SCW_\rho(\mu,\nu)
    =\max\{\operatorname{CTV}(\mu,\nu),\operatorname{CTV}(\nu,\mu)\}.
\end{align*}
In particular,
\[
    TV(\mu,\nu)\leq\operatorname{CTV}(\mu,\nu)
    \leq\operatorname{SCTV}(\mu,\nu)\leq ATV(\mu,\nu).
\]
Our first main result is a dimension-free bound for causal total variation,
proved in \Cref{thm:causalPinsker}.
\begin{theorem}\label{thm:introCausalPinsker}
The symmetrized causal total variation satisfies
\begin{align}\label{eq:symCausalPinsker}
    \operatorname{SCTV}(\mu,\nu)
    \leq2\sqrt{2H(\mu|\nu)}.
\end{align}
\end{theorem}
The factor $2$ in \eqref{eq:symCausalPinsker} is optimal for each causal
orientation, and hence also for their maximum, uniformly over all time
horizons; see \Cref{thm:CTVSharpConstant}.

Our second main result is the following adapted Pinsker inequality:
\begin{theorem}\label{thm:MainTheorem}
  Assume that $\mu, \nu$ are probabilities on the Polish space $\X_1\times \ldots \times \X_n$. Then
  \begin{align}\label{eq:AP}
  ATV(\mu, \nu)\leq \sqrt{n}\sqrt{2H(\mu|\nu)}.\end{align}
\end{theorem}
As in the classical case, this inequality is tight for every $n$, see \Cref{cor:PinskerTight} below.

In both cases we construct a natural optimal coupling explicitly.

We provide some context on the adapted Pinsker inequality \eqref{eq:AP}. 

The adapted total variation was first studied by Eckstein--Pammer \cite{EcPa22}, who establish the bound $ATV(\mu, \nu)\leq (2^n -1) TV(\mu, \nu)$. In particular the two distances are equivalent. More recently, Acciaio, Hou and Pammer \cite[Theorem 2.5 and Remark 2.6]{AcHoPa25} improved this to the sharp bound $ATV(\mu,\nu)\leq (2n-1)TV(\mu,\nu)$.
This is remarkable since, in general, the usual Wasserstein distance $W$ and its adapted counterpart $\AW$ are not equivalent: $\AW$ generates a strictly finer topology.
The idea of bounding an adapted distance by its classical counterpart $W$ under appropriate regularity conditions is put center stage in the work of Blanchet, Larsson, Park and Wiesel \cite{BlLaPaWi24}. In particular, for probabilities on $(\R^d)^n$ with common bounded support and uniformly Lipschitz one-step kernels, they establish a bound of the form $\AW_1(\mu,\nu)\leq C\sqrt{W_1(\mu,\nu)}$, where $C$ depends on the number of periods, the dimension, the support and the Lipschitz constant.

After the first version of this note, which treated only the bicausal case,
Park \cite{Pa26} established a transport--entropy inequality for weighted
adapted total variation. His result yields an adapted Pinsker inequality
with the same square-root dependence on the number of periods, as well as
adapted $T_1$ inequalities under a quadratic exponential moment assumption
on the reference measure.

It can be expected that the chain rule of entropy is well suited to control adapted variants of the Wasserstein distance. To give an example in this direction, recall that a probability $\gamma$ is said to satisfy a $T_2$-inequality with constant $C$ if the quadratic Wasserstein distance $W_2$ satisfies $W_2(\mu, \gamma) \leq \sqrt{2C H(\mu|\gamma)}$ for all probabilities $\mu$. From the chain rule of entropy it follows easily that $T_2$ inequalities tensorize perfectly, i.e.\ $W_2(\mu, \otimes_{k=1}^n\gamma) \leq \sqrt{2C H(\mu|\otimes_{k=1}^n\gamma)}$, see e.g.\ \cite[Proposition 22.5]{Vi09}. Indeed this argument actually implies that the same inequality holds also for the adapted Wasserstein distance, i.e.\ $\AW_2(\mu, \otimes_{k=1}^n\gamma) \leq \sqrt{2C H(\mu|\otimes_{k=1}^n\gamma)}$, see \cite[Section 5.2]{BaBeLiZa17} for details. Strikingly, $T_2$-inequalities exist also for adapted Wasserstein distances in continuous time (with the Wiener measure as reference measure), see the works of Föllmer \cite{Fo22a} and Lassalle \cite{La18}.

\section{Proofs}
Write $x_{1:k}:=(x_1,\ldots,x_k)$ and $\X_{1:k}:=\X_1\times\cdots\times\X_k$,
and let $\mu_{1:k},\nu_{1:k}$ denote the corresponding marginals.
We use $\mu^{x_{1:k}}$ and $\nu^{x_{1:k}}$ for the conditional laws of the next
coordinate. Thus the successive disintegrations take the form
\begin{align*}
    \mu(\dd x_{1:n})
    &=\mu_1(\dd x_1)\prod_{k=1}^{n-1}\mu^{x_{1:k}}(\dd x_{k+1}),\\
    \pi(\dd x_{1:n},\dd y_{1:n})
    &=\pi_1(\dd x_1,\dd y_1)
      \prod_{k=1}^{n-1}\pi^{x_{1:k},y_{1:k}}(\dd x_{k+1},\dd y_{k+1}).
\end{align*}
By \cite[Proposition 5.1]{BaBeLiZa17}, a coupling is bicausal precisely when
its successive conditional laws satisfy
\begin{align}\label{eq:iteratechar}
    \pi_1\in\cpl(\mu_1,\nu_1),\qquad
    \pi^{x_{1:k},y_{1:k}}\in\cpl(\mu^{x_{1:k}},\nu^{y_{1:k}}),
    \quad k<n,\quad\pi\text{-a.s.}
\end{align}
This characterization allows us to construct bicausal couplings by coupling
the one-step kernels successively.

For fixed marginals, $\cplc(\mu,\nu)$ and $\cplbc(\mu,\nu)$ are weakly
compact; see \cite[Lemma 2.1 and proof of Theorem 3.1]{BeLa20}.
Since $\rho$ is lower semicontinuous, both minima are attained by
the usual compactness argument. The representations below identify their
values and give explicit optimal couplings.

For finite measures $\lambda,\eta$, their common part $\lambda\wedge\eta$ is
defined by
\[
    \frac{\dd(\lambda\wedge\eta)}{\dd(\lambda+\eta)}
    =\min\Big\{\frac{\dd\lambda}{\dd(\lambda+\eta)},
                 \frac{\dd\eta}{\dd(\lambda+\eta)}\Big\}.
\]
If $\lambda,\eta$ are probabilities, its mass is
$|\lambda\wedge\eta|=1-\tfrac12 TV(\lambda,\eta)$; this is the largest
possible diagonal mass of a coupling.
For definiteness, we use the maximal coupling
\begin{align}\label{eq:maximalCouplingChoice}
    \mathsf M(\lambda,\eta)(\dd u,\dd v)
    :=(\lambda\wedge\eta)(\dd u)\delta_u(\dd v)
    +\frac{\big(\lambda-(\lambda\wedge\eta)\big)(\dd u)\big(\eta-(\lambda\wedge\eta)\big)(\dd v)}{1-|\lambda\wedge\eta|},
\end{align}
where the second term is zero if $|\lambda\wedge\eta|=1$. This couples the common part
on the diagonal and the remaining mass independently.
We define a natural bicausal coupling $\pi^{\mathrm{bc}}$ by starting with
$\mathsf M(\mu_1,\nu_1)$ and, conditionally on $(x_{1:k},y_{1:k})$,
using $\mathsf M(\mu^{x_{1:k}},\nu^{y_{1:k}})$ for the next coordinates.
This coupling is bicausal by \eqref{eq:iteratechar}.

We use the conventions $\mu^{x_{1:0}}:=\mu_1$ and
$\nu^{x_{1:0}}:=\nu_1$, and let $\beta_0$ be unit mass on the singleton
$\X_{1:0}$. Define the common submeasures recursively by
\begin{align}\label{eq:bicausalCommonMeasures}
    \beta_k(\dd x_{1:k})
    :=\beta_{k-1}(\dd x_{1:k-1})
    \big(\mu^{x_{1:k-1}}\wedge\nu^{x_{1:k-1}}\big)(\dd x_k),
    \qquad 1\leq k\leq n.
\end{align}
The measure $\beta_k$ describes the mass that can remain coupled on the
diagonal through time $k$ when the one-step kernels are coupled maximally.
In particular, $\beta_k\leq\mu_{1:k}\wedge\nu_{1:k}$.
The following representation was observed independently by Acciaio, Hou and
Pammer; see \cite[Lemma 3.2]{AcHoPa25}.

\begin{lemma}\label{le:ATVsimple}
The adapted total variation satisfies
\begin{align}\label{eq:ATVcommonMeasures}
    ATV(\mu,\nu)
    =2\big(1-|\beta_n|\big)
    =\sum_{k=1}^n\int
    TV(\mu^{x_{1:k-1}},\nu^{x_{1:k-1}})\,\dd\beta_{k-1}(x_{1:k-1}).
\end{align}
The bicausal coupling $\pi^{\mathrm{bc}}$ attains the infimum defining
$ATV(\mu,\nu)$ and puts $\beta_k$ on the diagonal up to time $k$.
\end{lemma}

\begin{proof}
For illustration, consider two time steps. Once $x_1\neq y_1$, the cost of
the whole path is already $1$, irrespective of the second coordinates.
Only mass that remains on the first diagonal can contribute to agreement
of the whole paths. Coupling maximally first at time $1$ and then at time
$2$ gives
\[
    ATV(\mu,\nu)=TV(\mu_1,\nu_1)
    +\int TV(\mu^{x_1},\nu^{x_1})\,\dd(\mu_1\wedge\nu_1)(x_1).
\]
The general case follows the same idea.

For a bicausal coupling $\pi$, let $\gamma_k$ be the law of $x_{1:k}$
restricted to $\{x_{1:k}=y_{1:k}\}$. At time $1$, we have
$\gamma_1\leq\mu_1\wedge\nu_1=\beta_1$. Given a common history $x_{1:k}$,
\eqref{eq:iteratechar} says that the next-step coupling has marginals
$\mu^{x_{1:k}}$ and $\nu^{x_{1:k}}$. Its diagonal part is therefore bounded
by their common part. Inductively,
\[
    \gamma_{k+1}
    \leq\gamma_k\otimes\big(\mu^{x_{1:k}}\wedge\nu^{x_{1:k}}\big)
    \leq\beta_{k+1}.
\]
Thus no bicausal coupling can put more than $|\beta_n|$ mass on the full
diagonal. For $\pi^{\mathrm{bc}}$, the maximal coupling
\eqref{eq:maximalCouplingChoice} gives $\gamma_k=\beta_k$ at every step.
The explicit formula \eqref{eq:maximalCouplingChoice} gives measurable
coupling kernels.
This proves the first equality in \eqref{eq:ATVcommonMeasures} and optimality
of $\pi^{\mathrm{bc}}$.

To obtain the sum in \eqref{eq:ATVcommonMeasures}, note that the mass lost
at time $k$ is
\begin{align}\label{eq:bicausalMassLoss}
    2\big(|\beta_{k-1}|-|\beta_k|\big)
    =\int TV(\mu^{x_{1:k-1}},\nu^{x_{1:k-1}})\,\dd\beta_{k-1}.
\end{align}
Summing \eqref{eq:bicausalMassLoss} and using $|\beta_0|=1$ gives the second
equality in \eqref{eq:ATVcommonMeasures}.
\end{proof}

\begin{proof}[Proof of \Cref{thm:MainTheorem}]
It suffices to consider $H(\mu|\nu)<\infty$. Write
$h_k(x_{1:k-1}):=H(\mu^{x_{1:k-1}}|\nu^{x_{1:k-1}})$.
The entropy chain rule gives
\begin{align}\label{eq:bicausalEntropyChain}
    H(\mu|\nu)=\sum_{k=1}^n\int h_k\,\dd\mu_{1:k-1}.
\end{align}
Apply classical Pinsker to each one-step term in
\eqref{eq:ATVcommonMeasures}, then Cauchy--Schwarz jointly to the sum and
the integrals:
\begin{align*}
    ATV(\mu,\nu)
    &\leq\sum_{k=1}^n\int\sqrt{2h_k}\,\dd\beta_{k-1}\\
    &\leq\Big(\sum_{k=1}^n|\beta_{k-1}|\Big)^{1/2}
      \Big(2\sum_{k=1}^n\int h_k\,\dd\beta_{k-1}\Big)^{1/2}
    \leq\sqrt{2nH(\mu|\nu)}.
\end{align*}
The last step uses $|\beta_{k-1}|\leq1$ and
$\beta_{k-1}\leq\mu_{1:k-1}$, followed by \eqref{eq:bicausalEntropyChain}.
The factor $\sqrt n$ comes from Cauchy--Schwarz in the sum over times.
\end{proof}

For sharpness, put
$b_\varepsilon:=(\frac12+\varepsilon)\delta_0+(\frac12-\varepsilon)\delta_1$
for $0\leq\varepsilon<1/2$.

\begin{lemma}\label{tightpinsker}
The classical Pinsker inequality is tight.
\end{lemma}

\begin{proof}
We have $TV(b_\varepsilon,b_0)=2\varepsilon$. Expanding the logarithms at zero,
\begin{align}\label{eq:BernoulliPinsker}
    H(b_\varepsilon|b_0)
    &=(\tfrac12+\varepsilon)\log(1+2\varepsilon)
      +(\tfrac12-\varepsilon)\log(1-2\varepsilon)\notag\\
    &=2\varepsilon^2+o(\varepsilon^2).
\end{align}
Hence $TV(b_\varepsilon,b_0)/\sqrt{2H(b_\varepsilon|b_0)}\to1$ as
$\varepsilon\downarrow0$.
\end{proof}

\begin{corollary}\label{cor:PinskerTight}
The adapted Pinsker inequality \eqref{eq:AP} is tight for every $n\in\N$.
\end{corollary}

\begin{proof}
Fix $n$ and take $\mu=b_\varepsilon^{\otimes n}$ and $\nu=b_0^{\otimes n}$.
All conditional laws are the corresponding one-period marginals, and
$|b_\varepsilon\wedge b_0|=1-\varepsilon$. Thus
$|\beta_n|=(1-\varepsilon)^n$. By \eqref{eq:ATVcommonMeasures},
\eqref{eq:BernoulliPinsker} and additivity of entropy for products,
\begin{align*}
    ATV(\mu,\nu)&=2\big(1-(1-\varepsilon)^n\big)
    =2n\varepsilon+O(\varepsilon^2),\\
    H(\mu|\nu)&=nH(b_\varepsilon|b_0)
    =2n\varepsilon^2+o(\varepsilon^2).
\end{align*}
Consequently $ATV(\mu,\nu)/\sqrt{2nH(\mu|\nu)}\to1$ as
$\varepsilon\downarrow0$. The same example embeds into any product whose
factors contain at least two points.
\end{proof}

\section{Dimension-free bounds for causal total variation}\label{sec:CTV}

Recall the definition of causal total variation:
\begin{align}\label{eq:defCTV}
    \operatorname{CTV}(\mu,\nu)
    =2\inf_{\pi\in\cplc(\mu,\nu)}\int\rho(x,y)\,\dd\pi(x,y)
    =2\big(1-\sup_{\pi\in\cplc(\mu,\nu)}\pi(\Delta_n)\big),
\end{align}
where $\Delta_t:=\{(x,y)\in\X\times\X:x_{1:t}=y_{1:t}\}$ is the diagonal
up to time $t$.
A coupling $\pi$ is causal precisely when it has a disintegration
\begin{align}\label{eq:causalSequential}
    \pi(\dd x_{1:n},\dd y_{1:n})
    =\mu(\dd x_{1:n})\prod_{t=1}^n
    K_t(x_{1:t},y_{1:t-1},\dd y_t).
\end{align}
Define $\alpha_1:=\mu_1\wedge\nu_1$ and, for $2\leq t\leq n$,
\begin{align}\label{eq:alpharec}
    \alpha_t:=\big(\alpha_{t-1}\otimes\mu^{x_{1:t-1}}\big)\wedge\nu_{1:t}.
\end{align}
Here $\alpha_{t-1}\otimes\mu^{x_{1:t-1}}$ means appending the next coordinate
according to the conditional law of $\mu$.
Unlike \eqref{eq:bicausalCommonMeasures}, \eqref{eq:alpharec} takes the common
part with the full marginal $\nu_{1:t}$. Causality prescribes the conditional
law of $x_t$ given $(x_{1:t-1},y_{1:t-1})$, while $y_t$ may depend on
$x_{1:t}$ and $y_{1:t-1}$.

We now define a natural causal coupling $\pi^{\mathrm c}$ by recursion.
Start with $\pi^1:=\mathsf M(\mu_1,\nu_1)$ from
\eqref{eq:maximalCouplingChoice}. Given $\pi^{t-1}$, append only the source
coordinate:
\[
    \bar\pi^t(\dd x_{1:t},\dd y_{1:t-1})
    :=\pi^{t-1}(\dd x_{1:t-1},\dd y_{1:t-1})
       \mu^{x_{1:t-1}}(\dd x_t).
\]
At each stage we reserve $\alpha_t$ for the diagonal. This is possible
inductively: $\pi^{t-1}$ puts at least $\alpha_{t-1}$ on the diagonal.
Extending by the conditional law of $\mu$ gives at least
$\alpha_{t-1}\otimes\mu^{x_{1:t-1}}$, which dominates $\alpha_t$ by
\eqref{eq:alpharec}.
Let $\gamma_t^0$ be the diagonal coupling of $\alpha_t$ with itself.
Its marginal on $(x_{1:t},y_{1:t-1})$ is dominated
by $\bar\pi^t$, and its target marginal is dominated by $\nu_{1:t}$.

Let $m_t$ be $\bar\pi^t$ minus the $(x_{1:t},y_{1:t-1})$-marginal of
$\gamma_t^0$, and set $\ell_t:=\nu_{1:t}-\alpha_t$. Their
$y_{1:t-1}$-marginals agree: both are
$\eta_t:=\nu_{1:t-1}-(\operatorname{pr}_{1:t-1})_\#\alpha_t$.
We couple these residuals independently conditional on the target history:
\[
    \gamma_t^1(\dd x_{1:t},\dd z,\dd y_t)
    :=\eta_t(\dd z)\,m_t^z(\dd x_{1:t})\,\ell_t^z(\dd y_t),
\]
where $z$ denotes $y_{1:t-1}$ and $m_t^z,\ell_t^z$ are the disintegrations
of $m_t,\ell_t$ with respect to $\eta_t$.
Set $\pi^t:=\gamma_t^0+\gamma_t^1$. Its marginals on
$(x_{1:t},y_{1:t-1})$ and $y_{1:t}$ are $\bar\pi^t$ and $\nu_{1:t}$.
Disintegrating with respect to $\bar\pi^t$ gives a kernel for $y_t$
depending only on $x_{1:t}$ and $y_{1:t-1}$; thus
\eqref{eq:causalSequential} shows that $\pi^t$ is a causal extension of
$\pi^{t-1}$. Finally, define $\pi^{\mathrm c}:=\pi^n$.

\begin{proposition}[Representation of $\operatorname{CTV}$]
\label{prop:CTVRepresentation}
The causal total variation satisfies
\begin{align}\label{eq:CTVRepresentation}
    \operatorname{CTV}(\mu,\nu)=2\big(1-|\alpha_n|\big).
\end{align}
The causal coupling $\pi^{\mathrm c}$ attains the infimum defining
$\operatorname{CTV}(\mu,\nu)$ and puts $\alpha_t$ on the diagonal up to time $t$.
\end{proposition}

\begin{proof}
For a causal coupling $\pi$, let $\beta_t^\pi$ be the law of $x_{1:t}$
restricted to $\Delta_t$, i.e.\
$\beta_t^\pi(B):=\pi(\Delta_t,\ x_{1:t}\in B)$. Clearly
$\beta_1^\pi\leq\alpha_1$. Causality means that, conditional on
$(x_{1:t-1},y_{1:t-1})$, the next source coordinate still has law
$\mu^{x_{1:t-1}}$. Hence
$\beta_t^\pi\leq\beta_{t-1}^\pi\otimes\mu^{x_{1:t-1}}$.
On $\Delta_t$, the source and target histories coincide, so the second
marginal also gives $\beta_t^\pi\leq\nu_{1:t}$. By induction,
\[
    \beta_t^\pi
    \leq\big(\beta_{t-1}^\pi\otimes\mu^{x_{1:t-1}}\big)\wedge\nu_{1:t}
    \leq\alpha_t.
\]
In particular, $\pi(\Delta_n)\leq|\alpha_n|$, which proves one direction
of \eqref{eq:CTVRepresentation}.

By construction, $\pi^{\mathrm c}$ puts at least $\alpha_t$ on $\Delta_t$.
The upper bound just proved forces equality.
In particular, $\pi^{\mathrm c}(\Delta_n)=|\alpha_n|$, so
\eqref{eq:defCTV} proves both \eqref{eq:CTVRepresentation} and optimality
of $\pi^{\mathrm c}$.
\end{proof}

The couplings $\pi^{\mathrm{bc}}$ and $\pi^{\mathrm c}$ thus put as much mass
as possible on the diagonal at each time. The residuals are coupled
independently, given both histories in the bicausal case and the target
history in the causal case. Optimal couplings need not be unique.

For $\mu\ll\nu$, let $Z_t:=\dd\mu_{1:t}/\dd\nu_{1:t}$ be the density
martingale, with $Z_0:=1$. Write $\overline Z_t:=\max_{0\leq s\leq t}Z_s$
and $\underline Z_t:=\min_{0\leq s\leq t}Z_s$, and put
$r_t:=\dd\alpha_t/\dd\nu_{1:t}$, with $r_0:=1$.
We choose the conditional densities $f_t:=Z_t/Z_{t-1}$ where $Z_{t-1}>0$,
and $f_t:=1$ otherwise. Since $Z$ is a nonnegative martingale,
\begin{align}\label{eq:conditionalDensityIdentities}
    Z_t=Z_{t-1}f_t,\qquad \E_\nu[f_t|\F_{t-1}]=1.
\end{align}

\begin{lemma}\label{lem:CTVLikelihoodFormulas}
If $\mu\ll\nu$, then
\begin{align}
    \tfrac12\operatorname{CTV}(\mu,\nu)
    &=\E_\nu\big[1-Z_n/\overline Z_n\big],\label{eq:CTVForwardMaxFormula}\\
    \tfrac12\operatorname{CTV}(\nu,\mu)
    &=\E_\nu[1-\underline Z_n].\label{eq:CTVReverseMinFormula}
\end{align}
\end{lemma}

\begin{proof}
Taking densities in \eqref{eq:alpharec}, extension by the conditional law of
$\mu$ multiplies $r_{t-1}$ by $f_t$, while taking the common part with
$\nu_{1:t}$ clips this density at $1$. Thus
\begin{align}\label{eq:causalDensityRecursion}
    r_t=(r_{t-1}f_t)\wedge1=Z_t/\overline Z_t.
\end{align}
For the second equality, start from $r_0=1$ and observe inductively that
$(Z_{t-1}f_t/\overline Z_{t-1})\wedge1
=(Z_t/\overline Z_{t-1})\wedge1=Z_t/\overline Z_t$.
Consequently $|\alpha_n|=\E_\nu[Z_n/\overline Z_n]$, and
\eqref{eq:CTVRepresentation} yields \eqref{eq:CTVForwardMaxFormula}.

For the reversed pair $(\nu,\mu)$, denote the corresponding common measures
by $\widetilde\alpha_t$ and their $\nu$-densities by $\widetilde r_t$.
Extension now uses the conditional law of $\nu$, so it leaves the density
unchanged; taking the common part with $\mu_{1:t}$ clips it at $Z_t$.
Starting with $\widetilde r_0=1$, we therefore obtain
\[
    \widetilde r_t=\widetilde r_{t-1}\wedge Z_t=\underline Z_t.
\]
Applying \eqref{eq:CTVRepresentation} to the reversed pair proves
\eqref{eq:CTVReverseMinFormula}.
\end{proof}

For a process $M$, write $\overline M_n:=\max_{0\leq t\leq n}M_t$.
We use the following standard form of Doob's maximal inequality.

\begin{lemma}[Doob's maximal inequality]\label{lem:DoobCrossing}
For every integrable submartingale $M$ and $b<a$,
\begin{align}\label{eq:DoobCrossing}
    \P(\overline M_n\geq a)\leq\frac{\E[(M_n-b)_+]}{a-b}.
\end{align}
\end{lemma}

\begin{proof}
Let $\tau$ be the first crossing time of $a$, stopped at $n$, and
$A:=\{\overline M_n\geq a\}$. On $A$, we have $M_\tau\geq a$.
Conditional Jensen and optional sampling give
$\E[(M_n-b)_+|\F_\tau]\geq(M_\tau-b)_+$.
Integrating over $A$ yields
\[
    \E[(M_n-b)_+]\geq\E[\mathbf{1}_A(M_\tau-b)_+]
    \geq(a-b)\P(A),
\]
which proves \eqref{eq:DoobCrossing}.
\end{proof}

\begin{theorem}[Causal Pinsker inequalities]
\label{thm:causalPinsker}\label{thm:reverseCausalPinsker}
For probabilities $\mu,\nu$ on $\X$,
\begin{align}
    \operatorname{CTV}(\mu,\nu)&\leq2\sqrt{2H(\mu|\nu)},\label{eq:causalPinsker}\\
    \operatorname{CTV}(\nu,\mu)&\leq2\sqrt{2H(\mu|\nu)}.\label{eq:reverseCausalPinsker}
\end{align}
\end{theorem}

\begin{proof}
Write $h:=H(\mu|\nu)$. The cases $h=0$ and $h=\infty$ are immediate, so
assume $0<h<\infty$. In particular, $\mu\ll\nu$.

For the forward bound \eqref{eq:causalPinsker}, the idea is that clipping
the extended densities at $1$ cannot increase the entropy accumulated along
the surviving diagonal mass. Put $\Psi(u):=u\log u-u+1$, with $0\log0:=0$, and
$J_t:=\E_\nu[\Psi(f_t)|\F_{t-1}]\geq0$.
Since $\E_\nu[f_t|\F_{t-1}]=1$, we have
$J_t=\E_\nu[f_t\log f_t|\F_{t-1}]$. Changing measure at time $t$ and using
$\prod_{t=1}^nf_t=Z_n$ gives the entropy chain rule in the form
\begin{align}\label{eq:entropyChainCausal}
    \sum_{t=1}^n\E_\nu[Z_{t-1}J_t]
    &=\sum_{t=1}^n\E_\mu[\log f_t]
      =\E_\mu[\log Z_n]\notag\\
    &=H(\mu|\nu)=h.
\end{align}
The choice $f_t=1$ where $Z_{t-1}=0$ makes $J_t=0$ there; finiteness of
\eqref{eq:entropyChainCausal} makes the remaining terms well defined.
Expanding $\Psi(r_{t-1}f_t)$ and using
\eqref{eq:conditionalDensityIdentities} gives the conditional identity
\begin{align}\label{eq:conditionalPsi}
    \E_\nu[\Psi(r_{t-1}f_t)|\F_{t-1}]
    =\Psi(r_{t-1})+r_{t-1}J_t.
\end{align}
Since $\Psi$ is minimized at $1$, clipping at $1$ decreases it. Also,
\eqref{eq:causalDensityRecursion} implies $r_{t-1}\leq Z_{t-1}$ because
$\overline Z_{t-1}\geq Z_0=1$. We conclude from \eqref{eq:conditionalPsi} that
\begin{align*}
    \E_\nu[\Psi(r_t)]
    &\leq\E_\nu[\Psi(r_{t-1})]+\E_\nu[r_{t-1}J_t]\\
    &\leq\E_\nu[\Psi(r_{t-1})]+\E_\nu[Z_{t-1}J_t].
\end{align*}
Summing these inequalities, starting from $\Psi(r_0)=0$, and using
\eqref{eq:entropyChainCausal} yields $\E_\nu[\Psi(r_n)]\leq h$.
Finally, $(1-u)^2\leq2\Psi(u)$ for $0\leq u\leq1$; this follows from
$\Psi''(u)=1/u\geq1$ and $\Psi(1)=\Psi'(1)=0$.
As $\tfrac12\operatorname{CTV}(\mu,\nu)=\E_\nu[1-r_n]$ by
\eqref{eq:CTVRepresentation}, Cauchy--Schwarz now gives
\[
    \big(\tfrac12\operatorname{CTV}(\mu,\nu)\big)^2
    \leq\E_\nu[(1-r_n)^2]\leq2\E_\nu[\Psi(r_n)]\leq2h.
\]

For the reverse bound \eqref{eq:reverseCausalPinsker}, it suffices by
\eqref{eq:CTVReverseMinFormula} to bound $\E_\nu[1-\underline Z_n]$.
We apply Doob's maximal inequality to the martingale $1-Z$.
For $a>1$ and $0<s<1$, use level $1-s$ and lower level $1-as$ in
\eqref{eq:DoobCrossing} to obtain
\begin{align}\label{eq:reverseLevelCrossing}
    \nu(\underline Z_n\leq s)
    \leq\frac{\E_\nu[(as-Z_n)_+]}{(a-1)s}.
\end{align}
To integrate \eqref{eq:reverseLevelCrossing}, note that for $0\leq z\leq a$
we have
\[
    \int_0^1\frac{(as-z)_+}{s}\,\dd s
    =\int_{z/a}^1\big(a-z/s\big)\,\dd s
    =a-z+z\log(z/a).
\]
For $z>a$, the integral vanishes, whereas the last expression is
nonnegative: as a function of $z$, it is convex and minimized at $z=a$.
It is therefore an upper bound for every $z\geq0$.
Tonelli's theorem and $\E_\nu[Z_n]=1$ now give
\begin{align}\label{eq:reverseCrossingBound}
    \tfrac12\operatorname{CTV}(\nu,\mu)
    &=\int_0^1\nu(\underline Z_n\leq s)\,\dd s\notag\\
    &\leq\frac{\E_\nu[a-Z_n+Z_n\log(Z_n/a)]}{a-1}
    =\frac{h+a-1-\log a}{a-1}.
\end{align}
Set $u:=a-1$. The elementary estimate
$u-\log(1+u)=\int_0^u t/(1+t)\,\dd t\leq u^2/2$ bounds the right-hand side
of \eqref{eq:reverseCrossingBound} by $h/u+u/2$.
Choosing $u=\sqrt{2h}$ proves
\eqref{eq:reverseCausalPinsker}.
\end{proof}


\section{Sharpness of the causal Pinsker constants}\label{sec:sharpCTV}

\begin{theorem}\label{thm:CTVSharpConstant}
The constant $2$ in \eqref{eq:causalPinsker}, \eqref{eq:reverseCausalPinsker}
and \eqref{eq:symCausalPinsker} is optimal uniformly over all time horizons.
\end{theorem}

\begin{proof}
We have proved the upper bounds. For sharpness in both directions, we use
two small changes of measure based on the same stopped martingale.

Fix $N\in\N$ and $\delta>0$, and put $q:=(1+\delta)^{-1}$ and $p:=1-q$.
Start with $M_0=0$. As long as the process has not jumped down, its value at
time $k$ is $k\delta$; it then moves to $(k+1)\delta$ with probability $q$
and to $k\delta-1$ with probability $p$. After a downward jump it remains
constant. Since $q\delta-p=0$, this defines a bounded martingale up to time
$N$. Before the downward jump, the conditional variance of its next increment is
$q\delta^2+p=\delta$.

Write $\overline M_N:=\max_{0\leq k\leq N}M_k$. The probability of making at
least $j$ upward steps is $q^j$, so summing these tail probabilities gives
its expected maximum. For the second moment, we use orthogonality of
martingale increments and note that the probability of not having jumped
down by time $k$ is $q^k$. Thus
\begin{align}\label{eq:terminalSecondMomentSharp}
\begin{aligned}
    \E[\overline M_N]&=\delta\sum_{j=1}^Nq^j=1-q^N,\\
    \E[M_N^2]&=\delta\sum_{k=0}^{N-1}q^k=(1+\delta)(1-q^N).
\end{aligned}
\end{align}

Let $\nu$ be the law of $(M_1,\ldots,M_N)$ and define $\mu_\varepsilon^\pm$
by $\dd\mu_\varepsilon^\pm/\dd\nu=1\pm\varepsilon M_N$.
As $-1\leq M_N\leq N\delta$, the choice
$0<\varepsilon<\min\{1,(N\delta)^{-1}\}$ makes both densities strictly
positive. They integrate to one because $\E[M_N]=0$.
The martingale property gives the likelihood processes
$Z_k^\pm=1\pm\varepsilon M_k$.

For the forward orientation, $\overline Z_N^+=1+\varepsilon\overline M_N$,
and \eqref{eq:CTVForwardMaxFormula} yields
\begin{align}\label{eq:sharpForwardFormula}
    \tfrac12\operatorname{CTV}(\mu_\varepsilon^+,\nu)
    &=\E_\nu\Big[1-\frac{1+\varepsilon M_N}
                              {1+\varepsilon\overline M_N}\Big]\notag\\
    &=\varepsilon\E_\nu\Big[\frac{\overline M_N-M_N}
                                  {1+\varepsilon\overline M_N}\Big].
\end{align}
For fixed $N$ and $\delta$, boundedness allows us to let
$\varepsilon\downarrow0$ inside the expectation. Since
$\E[\overline M_N-M_N]=\E[\overline M_N]$, the forward distance is
$2\varepsilon\E[\overline M_N]+o(\varepsilon)$.
For the reverse orientation, $\underline Z_N^-=1-\varepsilon\overline M_N$,
so \eqref{eq:CTVReverseMinFormula} gives the exact identity
$\operatorname{CTV}(\nu,\mu_\varepsilon^-)=2\varepsilon\E[\overline M_N]$.

Finally, the Taylor expansion of $(1+u)\log(1+u)$, together with
$\E[M_N]=0$, gives for either sign
\begin{align}\label{eq:sharpLinearization}
    H(\mu_\varepsilon^\pm|\nu)
    &=\E_\nu\big[(1\pm\varepsilon M_N)\log(1\pm\varepsilon M_N)\big]\notag\\
    &=\tfrac12\varepsilon^2\E[M_N^2]+o(\varepsilon^2).
\end{align}
Combining \eqref{eq:terminalSecondMomentSharp}--\eqref{eq:sharpLinearization},
the normalized ratios in the two orientations have the same limit:
\begin{align}\label{eq:sharpRatios}
\begin{aligned}
    \lim_{\varepsilon\downarrow0}
    \frac{\operatorname{CTV}(\mu_\varepsilon^+,\nu)}
         {\sqrt{2H(\mu_\varepsilon^+|\nu)}}
    &=2\sqrt{\frac{1-q^N}{1+\delta}},\\
    \lim_{\varepsilon\downarrow0}
    \frac{\operatorname{CTV}(\nu,\mu_\varepsilon^-)}
         {\sqrt{2H(\mu_\varepsilon^-|\nu)}}
    &=2\sqrt{\frac{1-q^N}{1+\delta}}.
\end{aligned}
\end{align}

Now choose $\delta_N=N^{-1/2}$ and $q_N=(1+\delta_N)^{-1}$.
Then $\delta_N\to0$ and $q_N^N\to0$, so the right-hand sides of
\eqref{eq:sharpRatios} tend to $2$.
For each $N$, take $\varepsilon_N$ sufficiently small that both ratios are
within $1/N$ of their limits in \eqref{eq:sharpRatios}.
These examples on $\R^N$ prove sharpness in both directions.
Since $\operatorname{SCTV}(\mu,\nu)\geq\operatorname{CTV}(\mu,\nu)$,
the same examples also prove sharpness for the symmetrized distance.
\end{proof}

For one time step, causal total variation coincides with total variation,
and the optimal constant is $1$; the constant becomes $2$ when the estimate
is required uniformly over time horizons.

\medskip
\noindent
{\bf Acknowledgment:}
 This research was funded  in whole or in part by the Austrian Science Fund (FWF) [doi: 10.55776/P35197]. For open access purposes, the author has applied a CC BY public copyright license to any author-accepted manuscript version arising from this submission.	
    
\enlargethispage{2\baselineskip}

\end{document}